\documentclass[12pt,a4paper,reqno]{amsart}
\usepackage{amsmath,amsfonts,amssymb, amsthm}

\usepackage{color}
\usepackage{mdwlist}
\usepackage{enumerate}
\usepackage{amscd}
\usepackage[all,cmtip]{xy}
\usepackage{fancyhdr}

\newcommand{\p}{\ensuremath{\mathfrak{p}}}
\newcommand{\q}{\ensuremath{\mathfrak{q}}}

\newtheorem{thm}{Theorem}[section]
\newtheorem*{thm*}{Theorem}
\newtheorem{lm}[thm]{Lemma}
\newtheorem{prop}[thm]{Proposition}

\theoremstyle{definition}\newtheorem{defi}[thm]{Definition}
\theoremstyle{definition}\newtheorem{ex}[thm]{Example}
\theoremstyle{remark}\newtheorem{rem}[thm]{Remark}

\DeclareMathOperator{\he}{ht}

\DeclareMathOperator{\trdeg}{tr.deg}

\DeclareMathOperator{\spec}{Spec}

\DeclareMathOperator{\codim}{codim}

\newcommand{\ra}{\rightarrow}
\newcommand{\func}[3]{\ensuremath{#1\mathpunct: #2\ra #3}}

\newcommand{\injfunc}[3]{\ensuremath{#1\mathpunct: #2\hookrightarrow #3}}

\begin{document}

\title[Some remarks on biequidimensionality]{Some remarks on
  biequidimensionality of topological spaces and Noetherian schemes}
\author{Katharina Heinrich} \address{KTH Royal Institute of
  Technology, Institutionen f\"or matematik, 10044 Stockholm, Sweden}
\email{kchal@math.kth.se} 

\begin{abstract}
  There are many examples of the fact that dimension and codimension
  behave somewhat counterintuitively. In EGA it is stated that a
  topological space is equidimensional, equicodimensional and catenary
  if and only if every maximal chain of irreducible closed subsets has
  the same length. We construct examples that show that this is not
  even true for the spectrum of a Noetherian ring. This gives rise to
  two notions of biequidimensionality, and we show how these relate to
  the dimension formula and the existence of a codimension function.
\end{abstract}

\maketitle
\thispagestyle{empty}

\section{Introduction}
Unless otherwise stated, all topological spaces considered here are
Noetherian $T_0$-spaces of finite dimension.

For a topological space $X$ we define its {\it Krull dimension} and
codimension in terms of maximal chains of irreducible closed
subsets. Then the following definitions are standard, see for example
\cite[D\'efinition (0.14.1.3), D\'efinition (0.14.2.1) and Proposition
(0.14.3.2)]{EGAIV1}.
\begin{defi}
  Let $X$ be a topological space.
  \begin{enumerate}
  \item The space $X$ is {\em equidimensional} if all irreducible
    components of $X$ have the same dimension.
  \item The space $X$ is {\em equicodimensional} if all minimal
    irreducible closed subsets of $X$ have the same codimension in
    $X$.
  \item The space $X$ is {\em catenary} if for all irreducible closed
    subsets $Y\subseteq Z$ all saturated chains of irreducible
    closed subsets that start with $Y$ and end in $Z$ have the same
    length.
  \end{enumerate}
\end{defi}
In addition, we define the following.
\begin{defi}\label{defi:equidim}
  Let $X$ be a topological space.
  \begin{enumerate}
  \item\label{item:defi_w_equidim} The space $X$ is {\em weakly
      biequidimensional} if it is equidimensional, equicodimensional
    and catenary.
   \item\label{item:defi_equidim} The space $X$ is {\em
       biequidimensional} if all maximal chains of irreducible closed
     subsets of $X$ have the same length.
  \end{enumerate}
\end{defi}
We will see in Lemma~\ref{lm:1_to_2} that every biequidimensional
space is weakly biequidimensional. In \cite[Proposition
(0.14.3.3)]{EGAIV1} it is moreover claimed that a topological space is
equidimensional, equicodimensional and catenary if and only if all
maximal chains have the same length. Furthermore, they define a space
to be biequidimensional if ``those equivalent properties'' hold. In
Section~\ref{sec:counterex}, however, we construct examples that show
that this is not the case even for Noetherian affine schemes.

The results on biequidimensional schemes stated in EGA are correct as
long as biequidimensional is defined in the stronger sense. In
Section~\ref{sec:dimension-formula} we show for example that the {\em
  dimension formula} \cite[Corollaire (0.14.3.5)]{EGAIV1} need not
hold for weakly biequidimensional schemes. Moreover, we prove in
Section~\ref{sec:codim-function} that every biequidimensional space
has a codimension function whereas a weakly biequidimensional space
need not.

The main reference for biequidimensional spaces and schemes is
\cite{EGAIV1}. In accordance with the incorrect equivalence
\cite[Proposition (0.14.3.3)]{EGAIV1}, many articles define
biequidimensional as equidimensional, equicodimensional and catenary,
and then they use properties like the dimension formula, that only
hold for spaces or schemes that are biequidimensional in the stronger
sense. In most cases the spaces considered are even biequidimensional
in the stronger sense, so the damage is relatively small. The purpose
of this article is to raise awareness of the difference between the
two concepts.

\subsection*{Acknowledgments.} I thank David Rydh for
reading this manuscript and giving valuable suggestions for
improvement.

\section{Biequidimensionality}\label{sec:biequidimensionality}
Before constructing the examples, we discuss the two notions of
bi\-equi\-dimensionality. As advertised earlier, we show first that
the property weakly biequidimensional is indeed weaker.
\begin{lm}\label{lm:1_to_2}
  Let $X$ be a topological space. If $X$ is biequidimensional, then
  $X$ is weakly biequidimensional.
\end{lm}
\begin{proof}
  Let $Z$ be an irreducible component in $X$. Every maximal chain of
  length $\dim(Z)$ in $Z$ is a maximal chain in $X$, and hence it has
  length $\dim(X)$. So $X$ is equidimensional. 

  In the same way it follows that every minimal irreducible closed
  subset has codimension $\dim(X)$, that is, the space $X$ is
  equicodimensional.

  Now let $Y\subseteq Z$ be two irreducible closed subsets in $X$. All
  saturated chains between $Y$ and $Z$ can be completed by the same
  irreducible sets to maximal chains in $X$, and hence they have the
  same length.
\end{proof}

As we will see in Section~\ref{sec:counterex}, the converse implication
does not hold. However, we have the following result.
\begin{lm}\label{lm:2-to-1_twist}
  Let $X$ be a topological space. If $X$ is equidimensional, catenary
  and every irreducible component of $X$ is equicodimensional, then
  $X$ is biequidimensional.
\end{lm}
\begin{proof}
  Let $X_0\subsetneq X_1\subsetneq \ldots \subsetneq X_k$ be a maximal
  chain of irreducible closed subsets in $X$. We have to show that it
  has length $\dim(X)$.

  Since $X$ is catenary, we have that $k=\codim(X_0,X_k)$. The minimal
  subset $X_0$ has codimension $\dim(X_k)$ in the equicodimensional
  component $X_k$. Finally we have that $\dim(X_k)=\dim(X)$ as $X$ is
  equidimensional.
\end{proof}
\begin{prop}\label{prop:biequidim_equiv}
  Let $X$ be a topological space. Then the following are equivalent.
  \begin{enumerate}
  \item\label{item:bieq1} The space $X$ is biequidimensional.
  \item\label{item:bieq2} The space $X$ is equidimensional, and for all
    irreducible closed subsets $Y\subseteq Z$ of $X$ we have
    that \begin{equation}\dim(Z) = \dim(Y) +
      \codim(Y,Z). \label{eq:bieq1} \end{equation}
  \item\label{item:bieq3} The space $X$ is equicodimensional, and for
    all irreducible closed subsets $Y\subseteq Z$ of $X$ we have
    that \begin{equation} \codim(Y,X) = \codim(Y,Z) +
      \codim(Z,X).\label{eq:bieq2}
\end{equation}
\end{enumerate}
\end{prop}
\begin{proof}
  These equivalences are shown in \cite[Proposition
  (0.14.3.3)]{EGAIV1}. For the sake of completeness we give a proof
  here as well.

  Suppose first that $X$ is biequidimensional. We showed in
  Lemma~\ref{lm:1_to_2} that $X$ is equidimensional, equicodimensional
  and catenary. Let $Y\subseteq Z$ be two irreducible closed
  subsets. All maximal chains in $Z$ can be extended by the same
  irreducible sets containing $Z$ to maximal chains in $X$, and hence
  they all have length $\dim(Z)$. In particular this holds for the
  chain obtained by composing a saturated chain of length $\dim(Y)$ in
  $Y$ and one of length $\codim(Y,Z)$ between $Y$ and $Z$, and
  Equation~(\ref{eq:bieq1}) follows. Equation~(\ref{eq:bieq2}) can be
  shown in the same way.

  For the converse implications, we show first that each of the
  Equations~(\ref{eq:bieq1}) and~(\ref{eq:bieq2}) implies that $X$ is
  catenary. Let $Y\subseteq Z \subseteq T$ be irreducible closed
  subsets in $X$.  Applying formulas~(\ref{eq:bieq1}) and
  (\ref{eq:bieq2}) respectively to all three inclusions $Y\subseteq
  Z$, $Z\subseteq T$ and $Y\subseteq T$ gives \[\codim(Y,T) =
  \codim(Y,Z)+\codim(Z,T).\] By \cite[Proposition (0.14.3.2)]{EGAIV1},
  this implies that $X$ is catenary.

  Now let $X_0\subsetneq X_1\subsetneq \ldots \subsetneq X_k$ be a
  maximal chain in $X$. Then $X_0$ is a closed point, and $X_k$ is an
  irreducible component. Suppose first that $X$ is as in
  assertion~\ref{item:bieq2}. Then $X$ is catenary, and hence we have
  that $k=\codim(X_0,X_k)$. Since the closed point $X_0$ has dimension
  $0$ and the irreducible component $X_k$ has dimension $\dim(X)$ by
  equidimensionality, Equation~(\ref{eq:bieq1}) applied to the
  inclusion $X_0\subsetneq X_k$ implies that $k=\dim(X)$.

  Similarly, in the case of assertion~\ref{item:bieq3} we have that
  $\codim(X_k,X)=0$ and, by equicodimensionality, that
  $\codim(X_0,X)=\dim(X)$. Hence Equation~(\ref{eq:bieq2}) implies
  that $k=\dim(X)$.
\end{proof}
Proposition~\ref{prop:biequidim_equiv} shows that assertions (a), (c)
and (d) in \cite[Proposition (0.14.3.3)]{EGAIV1} are equivalent, and by
Lemma~\ref{lm:1_to_2} they imply assertion~(b).

In the case that $X$ is a scheme, biequidimensionality can also be
described in the following way.
\begin{lm}\label{lm:scheme_bieq}
  Let $X$ be a scheme. Then the following are equivalent.
  \begin{enumerate}
  \item The scheme $X$ is biequidimensional.
  \item For every closed point $x\in X$ the local ring $\mathcal{O}_{X,x}$
    is catenary and equidimensional of dimension $\dim(X)$.
  \end{enumerate}
\end{lm}
\begin{proof}
  Let $x\in X$ be a point. Then there is an order-preserving bijection
  between the spectrum of the local ring $\mathcal{O}_{X,x}$ and the 
  irreducible closed subsets $Y$ of $X$ containing $\overline{\{x\}}$. In
  particular, we see that $X$ is catenary if and only if all local
  rings $\mathcal{O}_{X,x}$ are catenary. Moreover, a local ring
  $\mathcal{O}_{X,x}$ is equidimensional if and only if the subset
  $\overline{\{x\}}$ has the same codimension in every irreducible
  component of $X$ that contains it.

  Suppose first that $X$ is biequidimensional. Then $X$ and hence all
  local rings are catenary by Lemma~\ref{lm:1_to_2}. We have,
  moreover, that $\dim(\mathcal{O}_{X,x})=\codim(\overline{\{x\}},X)$ and,
  as $X$ is equicodimensional, for a closed point $x$ the latter
  number equals $\dim(X)$. Let $Z$ be an irreducible component
  containing the closed point $x$. Every saturated chain between
  $\{x\}$ and $Z$ is a maximal chain in $X$, and hence it has length
  $\dim(X)$. This shows that every closed point has the same
  codimension in every irreducible component containing it.

  Conversely, suppose that all local rings at closed points are
  catenary and equidimensional of dimension $\dim(X)$. Let
  $X_0\subsetneq \ldots \subsetneq X_k$ be a maximal chain in
  $X$. Then $X_0=\{x\}$ is a closed point, and there is an associated
  maximal chain in $\spec(\mathcal{O}_{X,x})$. Since the local ring
  $\mathcal{O}_{X,x}$ is catenary and equidimensional, it follows that
  $k=\dim(\mathcal{O}_{X,x})$. Hence all maximal chains in $X$ have
  length $\dim(X)$.
\end{proof}

\begin{rem}
  Note however that for a scheme to be biequidimensional it is not
  sufficient that it is equidimensional and that all local rings are
  catenary and equidimensional. Consider, for example, the spectrum
  $X$ of the localization of $k[u,v,w,x,y,z]/(wx)$ away from the union
  of the prime ideals $(u,v,w)$, $(w,x)$ and $(x,y,z)$. Then $X$ is
  equidimensional of dimension $2$. Being essentially of finite type
  over a field, the scheme $X$ and hence the local rings are catenary.
  The point corresponding to the maximal ideal $(w,x)$ has codimension
  $1$ in both irreducible components of $X$. As this is the only point
  that is contained in both components, we see that all local rings
  are equidimensional.  However, we have the following maximal chains
  of prime ideals in $X$:
  \begin{equation*} \xymatrix{(w) && (x) \\ (v,w)\ar@{-}[u] && (x,y)
      \ar@{-}[u] \\ (u,v,w)\ar@{-}[u] &
      (w,x)\ar@{-}[uur]\ar@{-}[uul] & (x,y,z)\rlap{.} \ar@{-}[u]}
  \end{equation*}
  Here the lines denote inclusion. We see that there are maximal chains
  of length $1$ and $2$. Hence $X$ is not biequidimensional.
\end{rem}
\begin{lm}\label{lm:fin_type}
  Let $X$ be an equidimensional scheme that is locally of finite type
  over a field $k$. Then $X$ is biequidimensional.
\end{lm}
\begin{proof}
  By Lemma~\ref{lm:2-to-1_twist}, it suffices to show that every
  irreducible component $Y$ of $X$ is catenary and equicodimensional.

  First we observe that all local rings $\mathcal{O}_{Y,y}$ are
  localizations of finitely generated $k$-algebras and hence, by
  \cite[Corollaire (0.16.5.12)]{EGAIV1}, they are catenary. As there
  is a bijection between the prime ideals in $\mathcal{O}_{Y,y}$ and the
  irreducible closed subsets of $Y$ containing $\overline{\{y\}}$, it
  follows that $Y$ is catenary.

  Let $\xi$ be the generic point in $Y$. By \cite[Proposition (5.2.1)
  and Equation (5.2.1.1)]{EGAIV2}, we have that
  $\dim(\mathcal{O}_{Y,y})=\trdeg_k(\kappa(\xi))=\dim(Y)$ for every
  closed point $y\in Y$.  It follows that all closed points have the
  same codimension in $Y$, that is, $Y$ is equicodimensional.
\end{proof}

\section{Construction of counterexamples}\label{sec:counterex}
In this section, we construct examples of topological spaces, affine
schemes and even Noetherian affine schemes that are weakly
biequidimensional but not biequidimensional.
\subsection{Topological spaces} Our first counterexample is the
following finite topological space.
\begin{ex}\label{ex:fliege_top}
  Consider the finite topological space $X$ with six points
  $x_1,\ldots,x_6$, each of them being the generic point of an
  irreducible closed subset. The relations between its irreducible
  closed subsets are given by the following diagram
  \[\xymatrix{\overline{\{x_5\}}\ar@{-}[d] & \overline{\{x_6\}}
    \ar@{-}[d] \\ \overline{\{x_3\}} \ar@{-}[d] & \overline{\{x_4\}}
    \ar@{-}[d]\\ \overline{\{x_1\}}\ar@{-}[uur] &
    \overline{\{x_2\}}\rlap{,}\ar@{-}[uul]}\]
  where a line between two sets denotes inclusion, with the sets
  increasing from bottom to top. Then the topological space $X$ is
  clearly equidimensional, equicodimensional and catenary. However,
  there are maximal chains of length $1$ as well as chains of length
  $2$.
\end{ex}

\subsection{Affine schemes}
So the properties biequidimensional and weakly bi\-equidimensional are
not equivalent, at least not in the given generality. This gives rise
to the question if they are equivalent at least in the particular case
that $X$ is the underlying topological space of an (affine) scheme.

However, even in this situation the answer is no. The following
theorem classifies {\em spectral spaces}, that is, topological spaces
that are the underlying topological space of an affine scheme.
\begin{thm}[{\cite[Theorem 6 and Proposition
    10]{Hochster}}]\label{thm:Hochster}
  Let $X$ be a topological space. The following are equivalent.
  \begin{enumerate}
  \item The space $X$ is isomorphic to the underlying topological
    space of the spectrum of a ring.
  \item The space $X$ is the projective limit of finite $T_0$-spaces.
  \end{enumerate}
\end{thm}
In particular this implies that every finite $T_0$-space is the
spectrum of a ring, and we get the following counterexample.
\begin{ex}\label{ex:fliege_scheme}
  Consider the topological space $X$ discussed in
  Example~\ref{ex:fliege_top}. By Theorem~\ref{thm:Hochster}, it can be
  realized as the spectrum of a ring. This gives an
  affine scheme that is weakly biequidimensional but not
  biequidimensional.
\end{ex}

So the equivalence does not hold for general (affine) schemes, but it
might still apply for Noetherian schemes.

\subsection{Noetherian schemes} The space defined in
Example~\ref{ex:fliege_top} cannot be realized as the spectrum of a
Noetherian ring by the following result.

\begin{prop}\label{prop:infinite_he1}
  Let $A$ be a Noetherian ring, and let $\p_1\subsetneq \p_2$ be two
  prime ideals such that $\he_{A/\p_1}(\p_2/\p_1)\geq 2$. Then there
  exist infinitely many prime ideals $\q$ with $\p_1\subsetneq \q
  \subsetneq \p_2$.
\end{prop}
\begin{proof}
  After replacing $A$ by $(A/\p_1)_{\p_2}$, we can, without loss of
  generality, assume that $A$ is a local domain of dimension $\geq 2$,
  and it suffices to show that $A$ has infinitely many prime ideals of
  height $1$. Note that by {\em Krull's Principal Ideal Theorem} every
  non-unit is contained in a prime ideal of height $1$. As the maximal
  ideal $\p_2$ consists of the set of non-units in $A$, it follows that
  $\p_2$ is contained in the union of all prime ideals of height
  $1$. If there are only finitely many prime ideals $\q_1,\ldots,\q_k$
  of height $1$, then {\em Prime avoidance} implies that $\p_2$ is
  contained in one of the $\q_i$, a contradiction.
\end{proof}

The question which topological spaces can be realized as $\spec(A)$
for a Noetherian ring $A$ is unfortunately still open; an extensive
survey of the state of the art can be found in \cite{Wiegand}. For our
particular case we can, however, make use of the following result.
\begin{thm}[{\cite[Theorem B]{DoeringLequain}}]\label{thm:emb}
  Let $\mathcal{A}$ be a finite partially ordered set. Then there exist a
  reduced Noetherian ring $A$ and an embedding $\injfunc i {\mathcal{A}}
  \spec(A)$ with the following properties.
  \begin{enumerate}
  \item The map $i$ establishes a bijection between the maximal
    (resp.\ minimal) elements of $\mathcal{A}$ and the maximal
    (resp.\ minimal) elements of $\spec(A)$.
  \item For all $a,a'\in \mathcal{A}$ such that $a<a'$ is saturated in
    $\mathcal{A}$, the chain $i(a)\subseteq i(a')$ is saturated in
    $\spec(A)$.
  \item For all $a,a'\in\mathcal{A}$, there exists a saturated chain of
    prime ideals of length $r$ between $i(a)$ and $i(a')$ if and only
    if there exists a saturated chain of length $r$ between $a$ and
    $a'$.
  \end{enumerate}
\end{thm}

\begin{ex}\label{ex:butterfly}
  Consider the finite partially ordered set $\mathcal{A}$ described by the
  diagram 
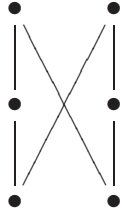
\[\xymatrix{\bullet\ar@{-}[d] & \bullet \ar@{-}[d] \\
    \bullet \ar@{-}[d] & \bullet \ar@{-}[d]\\ \bullet\ar@{-}[uur] &
    \bullet\ar@{-}[uul]\rlap{}}\] (where we again write elements in
  increasing order from bottom to top). There exist a Noetherian
  ring $A$ and an embedding $\injfunc i{\mathcal{A}} \spec(A)$ satisfying
  the properties of Theorem~\ref{thm:emb}. 

  Let us have a closer look at the difference between $i(\mathcal{A})$ and
  $\spec(A)$. Every additional point has to lie between a minimal and
  a maximal element but it cannot break a saturated chain. Moreover,
  we have to have infinitely many prime ideals of height 1 in every
  component of dimension 2 by
  Proposition~\ref{prop:infinite_he1}. This shows that the underlying
  topological space of $X=\spec(A)$ is given
  as 
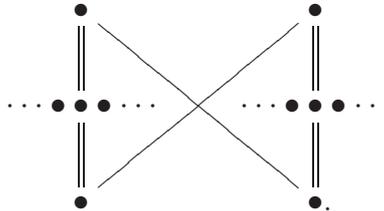
\[\xymatrix{\bullet\ar@{=}[d] & \bullet \ar@{=}[d] \\
    \cdots\bullet\bullet\bullet\cdots \ar@{=}[d] & \cdots
    \bullet\bullet \bullet \cdots \ar@{=}[d]\\ \bullet\ar@{-}[uur] &
    \bullet\ar@{-}[uul]\rlap{.}}\] Instead of having only two prime
  ideals of height $1$, there are infinitely many on each side. We see
  that $X$ is equidimensional, equicodimensional and catenary.
  However, there are maximal chains of length $1$ and of length $2$.
\end{ex}
This shows that even an affine Noetherian scheme that is weakly
biequidimensional need not be biequidimensional.

Another counterexample, which is moreover essentially of finite type
over a field, can be constructed in the following way.
\begin{ex}\label{ex:glue}
  Let $B= k[v,w,x,y]/(vy,wy)$. Then the spectrum $Y=\spec(B)$ is the
  scheme obtained by gluing $\mathbb{A}^3_k=\spec(k[v,w,x])$ and
  $\mathbb{A}^2_k=\spec(k[x,y])$ along the lines $v=w=0$ and $y=0$.
  Localizing $B$ away from the union $(v,w,x,y-1)\cup(v,w,y)$ of prime
  ideals gives a Noetherian ring $A$ with two minimal and two maximal
  ideals. The spectrum $X$ of $A$ looks
  like \[\xymatrix{(v,w)\ar@{-}[ddr]\ar@{=}[d] & (y) \ar@{=}[d] \\
    \cdots (v,w,x) \cdots\ar@{=}[d] & \cdots (w,y)\cdots \ar@{=}[d]\\
    (v,w,x,y-1) & (v,w,y)\rlap{,} } \] with infinitely many prime
  ideals between $(v,w)$ and $(v,w,x,y-1)$ and between $(y)$ and
  $(v,w,y)$.

  We observe that $X$ is equidimensional, equicodimensional and
  catenary.  However, we have the maximal chains $(v,w)\subsetneq
  (v,w,y)$ of length 1 and $(v,w) \subsetneq (v,w,x) \subsetneq
  (v,w,x,y-1)$ of length 2.

  Note that here, unlike in Example~\ref{ex:butterfly}, the ring is
  given explicitly.
\end{ex}

\begin{rem}
  The scheme constructed in Example~\ref{ex:glue} is essentially of
  finite type over a field. Note that there are no counterexamples
  that are locally of finite type over a field. In fact, by
  Lemma~\ref{lm:fin_type}, every equidimensional,
  and hence in particular every weakly biequidimensional, scheme
  locally of finite type over a field is biequidimensional.
\end{rem}

\section{The dimension formula}\label{sec:dimension-formula}
Next we show that the dimension formula holds in every
biequidimensional space. However, a modification of
Example~\ref{ex:glue} gives a weakly biequidimensional scheme where
the dimension formula does not hold.
\begin{prop}[{\cite[Corollaire (0.14.3.5)]{EGAIV1}}]
\label{prop:dimension-formula}
Let $X$ be a bi\-equi\-dimen\-sion\-al topological space. Then the
{\em dimension formula} holds for every irreducible closed subset $Y$
of $X$, that is, we have that
  \begin{align}\label{eq:dim_formula}
    \dim(X)=\dim(Y)+\codim(Y,X).
  \end{align}
\end{prop}
\begin{proof}
  Let $Y$ be an irreducible closed subset of $X$. We can choose
  maximal chains $Y_0\subsetneq Y_1\subsetneq\ldots \subsetneq Y_l=Y$ and
  \mbox{$Y=Y_0'\subsetneq Y_1'\subsetneq\ldots\subsetneq Y_k'$} of length
  $l=\dim(Y)$ and $k=\codim(Y,X)$ respectively. Then the composed
  chain $Y_0\subsetneq\ldots \subsetneq Y_l\subsetneq Y_1'\subsetneq
  \ldots \subsetneq Y_k'$
  is maximal. As $X$ is biequidimensional, this chain has length
  $\dim(X)$.
\end{proof}
Observe that the weakly biequidimensional Noetherian schemes that we
constructed in Examples~\ref{ex:butterfly} and \ref{ex:glue} satisfy
the dimension formula~(\ref{eq:dim_formula}). However, it does not in
general hold for weakly biequidimensional spaces as the following
modification of Example~\ref{ex:glue} shows.
\begin{ex}\label{ex:glue2}
  Consider $B = k[u,v,w,x,y,z] / (uy,uz,vy,vz,wy,wz)$. This is the
  coordinate ring of the scheme obtained by gluing the affine spaces
  $\mathbb{A}_k^4=\spec(k[u,v,w,x])$ and
  $\mathbb{A}_k^3=\spec(k[x,y,z])$ along the lines $u=v=w=0$ and
  $y=z=0$. Localization away from the union of the prime ideals
  $(u,v,w,y,z)$ and $(u,v,w,x,y-1,z-1)$ gives a Noetherian ring $A$ of
  pure dimension 3. The spectrum $X=\spec(A)$ is catenary, and it has
  two closed points corresponding to the prime ideals $(u,v,w,y,z)$
  and $(u,v,w,x,y-1,z-1)$, and both have codimension 3 in $X$. In
  particular, the scheme $X$ is weakly biequidimensional. Moreover, we
  have the following saturated chains of prime ideals in $X$:
  \[\xymatrix{  (u,v,w)\ar@{-}[d]&&(y,z)\ar@{-}[d]\\
    (u,v,w,x) \ar@{-}[d] && (w,y,z)\ar@{-}[d]\\
    (u,v,w,x,y-1) \ar@{-}[d] &&(v,w,y,z)\ar@{-}[d] \\
    (u,v,w,x,y-1,z-1) &&(u,v,w,y,z)\rlap{.} \ar@{-}[uuull]|-{\textstyle
      (u,v,w,y)}}\] We see that the prime ideal $\p=(u,v,w,y)$ has the
  property that $$\he(\p)+\dim(A/\p)=1+1=2<3=\dim(A),$$ that is, the
  dimension formula does not hold.
\end{ex}

\section{Existence of a codimension
  function}\label{sec:codim-function}
We conclude by showing that every biequidimensional topological space
has a codimension function whereas a weakly biequidimensional space
need not have it.
\begin{defi}[{\cite[Definition on p.~283]{Hartshorne:RD}}]
  Let $X$ be a topological space. A {\em codimension function} is a
  function $\func d X \mathbb{Z}$ such that \[d(x')=d(x)+1\] holds for every
  specialization $x'\in\overline{\{x\}}$ such that
  $\codim(\overline{\{x'\}},\overline{\{x\}})=1$.
\end{defi}

\begin{lm}\label{lm:loc_fin-codim-funct}
  \begin{enumerate}
  \item\label{item:loc_fin-codim-funct1} Let $X$ be a scheme locally
    of finite type over a field. Then
    {\[d(x)=-\dim(\overline{\{x\}})\]} defines a codimension function
    on $X$.
  \item Let $X$ be a scheme essentially of finite type over a field
    $k$. Then \[d(x)=-\trdeg_k(\kappa(x))\] defines a codimension
    function on $X$.
  \end{enumerate}
\end{lm}
\begin{proof}
  \begin{enumerate}
  \item We have to show that $\dim(\overline{\{x'\}}) =
    \dim(\overline{\{x\}})-1$ for all points $x,x'\in X$ such that
    $x'\in\overline{\{x\}}$ and
    $\codim(\overline{\{x'\}},\overline{\{x\}})=1$. The irreducible
    subscheme $\overline{\{x\}}$ is locally of finite type and hence
    biequidimensional by Lemma~\ref{lm:fin_type}. Then the statement
    is a direct consequence of the dimension formula, see
    Proposition~\ref{prop:dimension-formula}.
  \item Let $x,x'\in X$ be such that $x'\in\overline{\{x\}}$ and
    $\codim(\overline{\{x'\}},\overline{\{x\}})=1$. After, if
    necessary, replacing $X$ by an open affine neighborhood of $x'$,
    we can without loss of generality assume that $X=\spec(A)$, where
    $A=S^{-1}B$ for a finitely generated $k$-algebra $B$. Then $x$ and
    $x'$ correspond to points in $\spec(B)$ having residue fields
    $\kappa(x)$ and $\kappa(x')$. Hence it suffices to show that
    $d(x)=-\trdeg_k(\kappa(x))$ is a codimension function if $X$ is a
    scheme of finite type over $k$. In this case, we have by
    \cite[Proposition (5.2.1)]{EGAIV2} that
    \mbox{$\trdeg_k(\kappa(x))=\dim(\overline{\{x\}})$}, and $d(x)$ is
    the codimension function discussed in
    \ref{item:loc_fin-codim-funct1}.\qedhere
  \end{enumerate}
\end{proof}
\begin{prop}\label{prop:bieq-codim-funct}
  Let $X$ be a biequidimensional topological space. The
  map $$d(x)=\codim(\overline{\{x\}},X)$$ defines a codimension
  function on $X$.
\end{prop}
\begin{proof}
  The statement is a direct application of Equation~(\ref{eq:bieq2})
  in Proposition~\ref{prop:biequidim_equiv}.
\end{proof}
\begin{rem}
  Note that by Proposition~\ref{prop:dimension-formula} the
  codimension function in Proposition~\ref{prop:bieq-codim-funct} can
  be written as $d(x)=\dim(X)-\dim(\overline{\{x\}})$.  

  In particular, for a biequidimensional scheme locally of finite type
  over a field the codimension functions defined in
  Lemma~\ref{lm:loc_fin-codim-funct}\ref{item:loc_fin-codim-funct1}
  and in Proposition~\ref{prop:bieq-codim-funct} differ only by the
  constant term $\dim(X)$.
\end{rem}

Moreover, we have the following necessary and sufficient condition for
the function $d(x)=\codim(\overline{\{x\}},X)$ to be a codimension function.
\begin{prop}\label{prop:codim_func_local_rings}
  Let $X$ be a scheme. Then $d(x)=\codim(\overline{\{x\}},X)$ is a
  codimension function if and only if all local rings are catenary and
  equidimensional.
\end{prop}
\begin{proof}
  The existence of a codimension function directly implies that $X$
  and hence all the local rings are catenary.

  Suppose first that $d(x)=\codim(\overline{\{x\}},X)$ is a
  codimension function, and let $x\in X$. Let
  $\overline{\{x\}}=X_0\subsetneq \ldots \subsetneq X_k$ and
  $\overline{\{x\}}=X_0'\subsetneq \ldots \subsetneq X_l'$ correspond
  to two maximal chains in $\spec(\mathcal{O}_{X,x})$. Then the
  assumption on $d(x)$ implies that
  $\codim(\overline{\{x\}},X)=\codim(X_k,X)+k$ as well as that
  \mbox{$\codim(\overline{\{x\}},X)=\codim(X_l',X)+l$}. Both $X_k$ and
  $X_l'$ are irreducible components in $X$ and hence
  $k=\codim(\overline{\{x\}},X)=l$. This shows that the local ring
  $\mathcal{O}_{X,x}$ is equidimensional.

  For the converse implication, we assume that all local rings in $X$
  are catenary and equidimensional. Let $x'\in \overline{\{x\}}$ be a
  direct specialization. Let $\overline{\{x\}}=X_0\subsetneq \ldots
  \subsetneq X_k$ be a saturated chain of length
  $\codim(\overline{\{x\}},X)$. The extended chain
  $\overline{\{x'\}}\subsetneq X_0 \subsetneq \ldots \subsetneq X_k$
  corresponds to a maximal chain in $\spec(\mathcal{O}_{X,x'})$, and
  it has length
  \mbox{$\dim(\mathcal{O}_{X,x'})=\codim(\overline{\{x'\}},X)$} since
  the local ring $\mathcal{O}_{X,x'}$ is catenary and
  equidimensional. It follows that $\codim(\overline{\{x'\}},X) =
  \codim(\overline{\{x\}},X)+1$, that is, the function $d(x)$ is a
  codimension function.
\end{proof}
\begin{rem}
  The results of Proposition~\ref{prop:codim_func_local_rings} are
  stated in \cite[p.266]{Thorup} in the context of catenary gradings
  on schemes.
\end{rem}
As a consequence of Lemma~\ref{lm:loc_fin-codim-funct}, we see that
the schemes constructed in Examples~\ref{ex:glue} and ~\ref{ex:glue2}
do have codimension functions. The following discussion however shows
that a codimension function need not exist for a weakly
biequidimensional space.

\begin{ex}
  The scheme $X$ constructed in Example~\ref{ex:butterfly} does not
  have any codimension function. Consider the irreducible
  closed subsets $Z_1,\ldots,Z_6$ in $X$ satisfying the following
  inclusion relations 
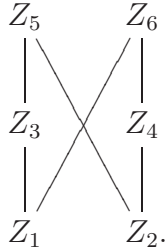
\[\xymatrix{Z_5\ar@{-}[d] & Z_6 \ar@{-}[d]
    \\ Z_3  \ar@{-}[d] & Z_4 \ar@{-}[d]\\
    Z_1\ar@{-}[uur] & Z_2\ar@{-}[uul]\rlap{.}}\] For $i=1,\ldots,6$
  let $z_i$ be the generic point of the irreducible set $Z_i$. 
Every codimension function $\func d X \mathbb{Z}$ would then have to satisfy
  $d(z_1)=d(z_5)+2=d(z_6)+1$ and $d(z_2)=d(z_5)+1=d(z_6)+2$, which is
  impossible.

  Note that $X$ is not only weakly biequidimensional but it satisfies
  the dimension formula. Still it does not have any codimension
  function.

  Furthermore, it follows from Lemma~\ref{lm:loc_fin-codim-funct} that
  $X$ is not essentially of finite type over a field.

  As a dualizing complex on a locally Noetherian scheme gives rise to
  a codimension function, see \cite[Proposition V.7.1]{Hartshorne:RD},
  we see moreover that $X$ does not have a dualizing complex.
\end{ex}

\end{document}